\documentclass[11pt]{amsart}
\usepackage{amsmath,amsfonts,amsthm,amssymb,amscd,colordvi}

\usepackage{color}
\usepackage{comment}

\newcommand{\p}{\partial}

\newcommand{\R}{{\mathbb R}}
\newcommand{\F}{{\mathbb F}}

\newcommand{\Z}{{\mathbb Z}}

\newcommand{\N}{{\mathbb N}}

\newcommand{\cL}{{\mathcal L}}

\newcommand{\supp}{\mathop{\rm supp}\nolimits}

\newcommand{\de}{\delta}
\newcommand{\dtx}{d_{\sigma_t(x)}}
\newcommand{\dtz}{
 dz\!\mid_{\Sigma_t}}
\newcommand{\NN}{{\mathbb N \cup \{0\}}}
\newcommand{\II}{\mathcal I}
\newcommand{\lappr}{\lesssim}
\newcommand{\dd}{d_1}

\def\12{\tfrac12}

\def\lan{\langle}
\def\ran{\rangle}

\def\eps{\varepsilon}
\theoremstyle{plain}
\newtheorem{theorem}{Theorem}[section]
\newtheorem{lemma}[theorem]{Lemma}
\newtheorem{amplification}[theorem]{Amplification}
\newtheorem{proposition}[theorem]{Proposition}
\newtheorem{corollary}[theorem]{Corollary}
\theoremstyle{definition}

\theoremstyle{remark}

\newtheorem{example}[theorem]{Example}

\newcommand{\be}{\begin{equation}}
\newcommand{\ee}{\end{equation}}
\newcommand{\lbl}{\label}
\newcommand{\non}{\nonumber}

\newcommand{\qu}{\quad}

\newcommand{\fr}{\frac}

\newcommand{\x}{\mathbf  x}
\newcommand{\y}{\mathbf  y}
\newcommand{\z}{\mathbf z}
\newcommand{\vu}{\mathbf u}
\newcommand{\vv}{\mathbf v}

\newcommand{\vb}{\mathbf b}
\newcommand{\vc}{\mathbf  c}
\newcommand{\vw}{\mathbf  w}
\newcommand{\va}{\mathbf  a}
\newcommand{\vl}{\mathbf  l}

\newcommand{\ds}{\displaystyle}

\title{Some remarks  on Heath-Brown's theorem on quadratic forms}
\author{Andrey  Dymov}
\address{Andrey Dymov \\ Steklov Mathematical Institute of RAS, Moscow 119991, Russia 
	\& National Research University Higher School of
	Economics, Moscow 119048, Russia} \email{dymov@mi-ras.ru}
\author{Sergei Kuksin}
\address{Sergei Kuksin \\ Universit\'e Paris-Diderot (Paris 7), UFR de Math\'ematiques - Batiment Sophie Germain, 5 rue Thomas Mann, 75205 Paris,
	France  \& School of Mathematics, Shandong University, Jinan, PRC }
\email{ Sergei.Kuksin@imj-prg.fr}

\author{ Alberto Maiocchi }
\address{Alberto Maiocchi \\ Universit\`a degli Studi di Padova,
  Dipartimento di Matematica, Padova, Italy} \email{alberto.maiocchi@unipd.it}

\author{Sergei  Vl\u adu\c t}
\address{Sergei  Vl\u adu\c t \\ Aix Marseille Universit\'e, CNRS, Centrale Marseille, I2M UMR 7373, 13453, Marseille,
France and IITP RAS, 19 B. Karetnyi, Moscow, Russia} \email{serge.vladuts@univ-amu.fr}

\begin{document}

\maketitle
\begin{abstract}
In his paper from 1996 on quadratic forms Heath-Brown developed a version of circle method to count points in the intersection of an unbounded quadric with a 
lattice of short period, if each point is given a weight. The weight
function is assumed to be $C_0^\infty$--smooth and to vanish near the
singularity of the  
quadric. In out work we allow the weight function to be finitely
smooth and not vanish near the singularity, and we  give also an
explicit dependence on the weight function.

\end{abstract}
\section{Introduction}

\subsection{Setting and result} \label{s_1.1}

Let us consider a non degenerate and  non sign-definite   quadratic form on $\R^d$,
\be\label{FA}
F(\z) = \tfrac12 A \z\cdot \z, 
\ee
where $A$ is a symmetric matrix. Then for $t\in \R$  the quadric
\be\label{st}
\Sigma_t =\{ \z: F^t(\z)=0\}, \quad  F^t = F-t, 
\ee
is an unbounded hyper-surface in $\R^d$. It is smooth if $t\ne0$, while $\Sigma_0$ is a cone and  has a locus at zero. 

Let $\Z^d_L$ be the lattice of a small period $L^{-1}$,
$$
\Z^d_L =L^{-1} \Z^d, \qquad L>1, 
$$
and let $w$ be a {\it  regular } real function on $\R^d$ which means that $w$ and its Fourier transform $\hat w(\xi)$ are continuous functions which 
 decay at infinity {  sufficiently fast:
\be\label{regular}
|w(\z)| \le C  |\z|^{-d-\gamma}, \qquad |\hat w( \mathbf \xi)| \le C  |\xi|^{-d-\gamma}, 
\ee
for some $\gamma>0$ and some $C>0$. }Our goal is to study the behaviour of  series 
$$
N_L(w; F, m)=  \sum_{\z \in \Sigma_m \cap \Z^d_L} 
      w(\z)\,, 
$$
where $m\in \R$ is such that   $mL^2$ is an integer.\footnote{E.g., $m=0$ -- this case is the most important for us.}
  Obviously,  
\be\label{obvious}
N_L(w{ ; F, m}) = N_1(w_L{ ;F, L^2m}) =: N(w_L{ ;F, L^2m}), \; \text{where}\;\; 
w_L(\z) := w(\z/L).  
\ee
To study $N_L(w{ ;F, L^2m})$ we closely follow the circle method in the form, given to it by  Heath-Brown in \cite{HB}.  We 
 start with a key theorem which expresses  the analogue of Dirac's
delta function on integers, i.e.  the function 
 $\delta:\Z\mapsto\R$ such that
\begin{equation*}
  \delta(n):= \left\{ \begin{array}{cc}
    1 & \mbox{for } n=0\\
    0 & \mbox{for } n\neq 0 \end{array} \right. \,,
\end{equation*}
through a sort of Fourier representation. This result goes back at least  to Duke, Friedlander and Iwaniec \cite{DFI} (cf. also \cite{I}) , and we  state it in the form, given in
 \cite[Theorem~1]{HB}; basically, it replaces (a major arc decomposition of) the trivial identity
$$\delta(n)= \int_0^1 e^{2\pi i\alpha n}d\alpha $$
employed in the usual circle method. In
the theorem for  $q\in \Z^*$ we denote by  $e_q$ the exponential function  $e_q(x):= e^{\tfrac{2\pi i x}{q}}$, and 
denote by ${\sum}^*_{a (\text{mod}\, q)} $ the summation over residues $a$ with
$(a,q)=1$, i.e., over all integers  $a\in [1, q-1]$, relatively prime with $q$. 
  
\begin{theorem}

\label{th:1}
  For any $Q>1$, there exists $c_Q>0$ and a smooth function 
  $ h(x,y): \R_{>0}\times
  \R\mapsto \R_{\ge 0}$,  such that
  \begin{equation}\label{gl}
    \delta(n) = c_QQ^{-2}\sum_{q=1}^{\infty}{\sum_{a(\!\!\!\!\!\!\mod
        q)}}^* e_q(an) h\left(\frac qQ,\frac n{Q^2}\right)\,.
  \end{equation}
  The constant $c_Q$ satisfies
  $\,
c_Q=1+O_N(Q^{-N})\,
$
for any $N>0$, while $h$ is such that 
$h(x,y)\le c/x$ and  $h(x,y)=0$ for $x>\max(1,2|y|)$ {  (so for each $n$ the sum in \eqref{gl} contains finitely many non zero terms). }
\end{theorem}

Since $N(\tilde w;F, t)$ may be written as 
$
\sum_{\z\in\Z^d} \tilde w(\z) \delta (F^t(\z)),
$
then Theorem~\ref{th:1} allows to represent series $N(\tilde w;F, t)$ as an iterated sum. Transforming that sum  further using the Poisson summation 
formula as in \cite[Theorem~2]{HB} we arrive at the following result:\,\footnote{In \cite{HB} the result below is stated for $w\in C_0^\infty$. But the argument there,
based on the Poisson summation,  applies as well to regular functions $w$.}

\begin{theorem}[Theorem 2 of \cite{HB}]\label{th:2}
  For any regular function $\tilde w$, any $t$ and any 
  $Q>1$ 
  we have the expression
  \begin{equation}\label{eq:sum}
    N({\tilde w}{ ; F, t})=c_QQ^{-2}\sum_{\vc\in
      \Z^{d}}\sum_{q=1}^\infty q^{- d} 
    S_q(\vc) I^0_q(\vc)\,,
  \end{equation}
  with
  \begin{equation}\label{eq:S}
S_q(\vc) :={\sum_{a(\!\!\!\!\!\!\mod
        q)}}^* \sum_{\vb(\!\!\!\!\!\!\mod q)} e_q(a{ F^t}(\vb) + \vc\cdot \vb)    
  \end{equation}
and
\begin{equation*}
I^0_q(\vc) := \int_{\R^{d}} {\tilde w}(\z)h\left(\frac qQ,
\frac{{ F^t}(\z)}{Q^2}\right) e_q(-\z\cdot \vc)
\,d\z\,.
\end{equation*}
\end{theorem}

We will apply  Theorem~\ref{th:2} to examine 
$N(w_L{ ;F, L^2 m}) = N_L(w{ ;F, m})$ with large $L$,  choosing $Q=L>1$ and estimating  
explicitly the leading terms in $L$ of $S_q(\vc)$ and $I^0_q(\vc)$ as well as the remainders. The answer will be given in terms of  the 
 integral
\be\label{lead_term}
\sigma_{\infty}(w; F, { m}) = \int_{ \Sigma_m}
w(\z)\,\mu^{ \Sigma_m}(d\z)\ 
\ee
(which is  singular if $m=0$). Here $\mu^{\Sigma_t}(d\z)=|A\z|^{-1}dz|_{\Sigma_t}$, with
$dz|_{\Sigma_t}$ 
representing the volume element over $\Sigma_t$, induced from the standard euclidean structure on $\R^{d}$,
and $A$ the symmetric matrix in \eqref{FA}. 
Another quantity, entering the asymptotic for $N_L(w{ ;F, m})$,
is  the infinite  product
  \be\label{77}
  \sigma(F,m) = \prod_p\sigma_p(F,m)\ ,
  \ee
  where $p$ ranges over all primes, and $\sigma_p$ is defined by
\begin{equation}\label{eq:sigma_p}
\sigma_p(F,m):=\sum_{l=0}^\infty p^{-dl}S_{p^l}(0),
\end{equation}
where $S_1\equiv 1$ and $S_q$ is given  by \eqref{eq:S} with $t=L^2m$.

Motivated by the continuation \cite{DKMV} of the research \cite{DK1} we are the most interested in the special case of the quadratic forms $F$
 when $d$ is an even number, so 
$
\R^d =\{ \z=(\x,\y) ,\; \x,\y \in \R^{d/2}\},
$
and 
\be\label{F0} 
F_0(\z) = \x\cdot \y, \quad A_0(\x,\y) = (\y,\x)\,. 
\ee 
Our main result, stated below, specifies  Theorem~5 from  \cite{HB}  for $F=F_0$  in three respects: firstly, now the function $w$ has  finite 
smoothness and sufficiently fast 
 decays at infinity, while in \cite{HB} $w\in C_0^\infty$, secondly, we specify how the remainder depends on $w$, and thirdly and the 
most importantly, we remove the imposed in \cite{HB} restriction that the support of $w$ does not contain the origin (this improvement is crucial for us 
since in \cite{DKMV} the theorem is used in the situation when $w(0) \ne0$).

We note that a similar  specification of the Heath-Brown method was obtained in 
 \cite[Section~5]{BGHS18}, also for  the purposes of wave turbulence.

In the theorem and everywhere  below for  a function $f\in C^k(\R^N)$ we denote 
$
\| f\|_{n_1, n_2}  = \sup_{\z\in \R^N} \max_{|\alpha|\le n_1} |\p^\alpha
f(\z)| \lan \z\ran^{n_2}\,,
$
where   $n_1 \in \NN$, $n_1\le k,$ and  $ n_2\in \R_{\ge 0}$.   Here
\begin{equation*}
  \langle \x\rangle := \max\{1,|\x|\} \quad \text{for}\;\: \x\in \R^l,
\end{equation*}
for any $l\ge1$. { 
 Note that if $\| w \|_{d+1, d+1} <\infty$, then the function $w$ is regular. Indeed, the fist relation in \eqref{regular} is obvious. To
prove the second note that for any integer vector 
$
\alpha \in (\N\cup\{0\})^d,
$
$
\xi^\alpha \hat w(\xi) = \big( \frac{i}{2\pi}\big)^{|\alpha|} \widehat{
\p_\x^\alpha w}(\xi).
$
But if 
$
|\alpha| \equiv \sum\alpha_j \le d+1,
$
then $|\p_\x^\alpha w| \le C \lan \x\ran^{-d-1}$, so $\p_\x^\alpha w$ is an $L_1$-function. Thus its Fourier transform 
 $ \widehat{\p_\x^\alpha w}$ is a bounded continuous function for each 
$|\alpha| \le d+1$ and the second relation in \eqref{regular} also holds. }

\begin{theorem}\label{th:3}
  For any $0<\eps\le1$ and any even dimension 
   $d> 4$, there exist constants $N_1(d,\eps)$ and  $N_2(d,\eps) \le N_3(d,\eps)$, 
  such that  if $w \in C^{N_1}$ and a real number $m$ satisfies $mL^2 \in \Z$,  then 
  \begin{equation}\label{m_result}
  \big|   N_L(w{ ; F_0, m})-
  \sigma_\infty(F_0,w{ , m})\sigma(F_0{,    m}) L^{d-2} \big|  \le C 
    L^{d/2+\eps}\left(\|w\|_{N_1,N_2}+\|w\|_{0,N_3}\right),
  \end{equation}
  where the constant $C$ depends on $d, \eps${ ,  $m$}. 
  In particular if $\eps=1/2$, then one can take  $N_1= 2d^2-2d$, $N_2=
  7(d+1)$ and $N_3=N_1+3d+4$.
\end{theorem}

\noindent{\it Remarks.} 
1) Here and everywhere else the dependence on $m$ is uniform
  on every compact interval. 
  \smallskip
  
  \noindent 2)
  The values of the constants $N_j(d, \eps)$ in \eqref{m_result}, obtained below, is far from optimal since
  it was not our goal to optimise them. 
    \smallskip
  
  \noindent 3) In the theorem   and in similar situations    below the result does not apply if the r.h.s. in \eqref{m_result} is infinite.
    \smallskip
  
    \noindent  {  4) Since the theorem's proof is based on the representation \eqref{eq:sum}, then the function $w$ should be regular
    (see \eqref{regular}). But this holds true if 
    $
    \| w\|_{d+1, d+1} <\infty,
    $
    and so is valid  if the r.h.s. of  \eqref{m_result}  
    is finite, with $N_1, N_2 $  sufficiently big. E.g. if $N_1, N_2 $ are  as big as in the last line of the theorem's assertion.
      \smallskip

\noindent    5) The quantity $\sigma(F_0{,    m})$ can be easily evaluated. Let us give this evaluation for $m=0$.  
Recall then that $\sigma(F_0{,    0})=\sigma(F_0)=\Pi_{prime\; p}\sigma_p(F_0)$ for 
$$ 
\sigma_p(d)=\sigma_p(F_0)= \lim_{k\to \infty} \frac{\sharp\{ F_0(x,y)\equiv 0\,\mod (p^{k})\}}{p^{(2d-1)k}}. 
 $$
However, since  for any prime $p$ the hypersurface $\{F_0 \, \mod\,p=0\}$ over $\F_p$ is the affine cone over a smooth projective quadric  we see that
 $$ \sigma_p(d)=  {\sharp\{ F_0(x,y) \equiv 0\,\mod (p )\}}p^{1-2d}=N_p(d)p^{1-2d}    $$
where $N_p$ is the number of $\F_p$-points on $\{F_0 \equiv 0\, \mod \,p\}.$\smallskip
Thus, $N_p(1)=2p-1$ and 
$$N_p(d+1)=\sharp \{ \hbox{\rm solutions with }   x_{d+1}=0\}+\sharp \{ \hbox{\rm solutions with } x_{d+1}\ne 0\}$$
$$ =pN_p(d )+(p-1)p^{2d}. \hskip 5 cm$$ 
Therefore, $$\forall d\ge 2,\; N_p(d)=p^{2d-1}+p^{d} -p^{d-1},\;\sigma_p(d)=1+p^{1-d}-p^{-d} .$$
For $\sigma_d$ we get then that
$$\sigma_3 =\prod_{  p}\big(1+p^{-2}-p^{-3}\big)=1.305..,\;\;\;\;
 \sigma_4 =\prod_{  p}\big(1+p^{-3}-p^{-4}\big)=1.100..,$$
whereas
$$ 1<\sigma_d =\prod_{  p}\big(1+p^{1-d}-p^{-d}\big ) <1+2^{2-d} $$
tends to 1 when $d$ grows.
  \medskip
  
  The proof of Theorem \ref{th:3} occupies the rest of the paper and closely follows that of \cite[Theorem 5]{HB} with additional control how the constants depend on $w$.
  The only significant difference comes in Sections 3 and 4 below where we do not assume that the function $w$ vanishes near the origin, the last assumption being crucial in the analysis of integrals in Sections 6 and 7 of \cite{HB}. To cope with this difficulty we have to  examine the smoothness at zero of the function 
  \be\label{hrum}
  m \mapsto \sigma_\infty (F_0, w, m)
  \ee
  and its decay at  infinity. The corresponding analysis is performed in Appendix, where using the techniques, developed in \cite{DK1} to study 
  integrals \eqref{lead_term}, we prove that  function \eqref{hrum} is $(d/2-2)$-smooth, but for 	a generic $w(z)$ its  derivative of  order  $(d/2-1)$ has a logarithmic 
  singularity ar zero. There we also get a (non optimal) estimate for the rate of growth of \eqref{hrum} at infinity. 
  \medskip

   The approach we use to prove Theorem \ref{th:3}     is general and applies to other quadratic forms \eqref{FA}.  Moreover, most of the 
 auxiliary results, obtained on the way to prove the  theorem, are established  for  the general  $F$. In an extended version \cite{refinement} 
  of this paper which is now  under preparation,   the theorem will be proven in its full generality. Namely, there we obtain
 
 \begin{amplification}
 The assertion of Theorem \ref{th:3} with modified constants $N_j$ and $C$
  remain true for any non degenerate and  non sign-definite   
   quadratic form \eqref{FA}. Then the constants $N_j$
  and $C$ also depend on the minimal and the maximal eigenvalues of the   operator $A$. 
 \end{amplification}

\subsection{Scheme of the proof of Theorem \ref{th:3}}\label{sec:2}

It is easy to see that if the r.h.s. in \eqref{m_result} is finite, then the function $w$ is regular in the sense of Section~\ref{s_1.1}, so Theorem~\ref{th:2} applies. 
Then, according to \eqref{eq:sum} and \eqref{obvious}, 
\begin{equation}\label{eq:sum'}
N_L(w{ ; F, m})=c_L \,L^{-2}\sum_{\vc\in \Z^{d}}\sum_{q=1}^\infty q^{-d}
S_q(\vc) I_q(\vc)\,,
\end{equation}
where $S_q$ is given by \eqref{eq:S} and 
\begin{equation}\label{eq:I'}
I_q(\vc) := \int_{\R^{d}} w\left(\frac{\z}{L}\right)h\left(\frac qL,
\frac{{ F^{mL^2}}(\z)}{L^2}\right) e_q(-\z\cdot \vc)
\,d\z\,.
\end{equation}
Consider 
$$
n(\vc)=n(\vc; L)=\sum_{q=1}^\infty q^{-d}
S_q(\vc) I_q(\vc), \; \mbox{so that}\; N_L(w{ ;F,m})=c_L L^{-2}\sum_{\vc\in\Z^{d}} n(\vc).
$$ 
For an $\gamma_1 \in (0,1)$ we  write
\begin{equation}\label{eq:N_L=}
N_L(w{ ;F, m})=c_L L^{-2}\big(J_0 + J_{<}^{\gamma_1} + J_{>}^{\gamma_1}\big),
\end{equation}
where 
\begin{equation}\label{eq:J><}
 J_0:=n(0)\,, \quad J_<^{\gamma_1}:=\sum_{\vc\ne 0,\,|\vc|\leq
   L^{\gamma_1}}  n(\vc) \,,\quad 
J_>^{\gamma_1}:=\sum_{|\vc|> L^{{\gamma_1}}}  n(\vc)\,.
\end{equation}
Proposition~\ref{l:19} (which is a modification of Lemmas~19 and  25  from \cite{HB})  implies that 
 $$
 |J_>^{\gamma_1}|\lappr_{d,{\gamma_1}{ ,m}} \|w\|_{N_0,2N_0+d+1}\,
 $$
with $N_0:=\lceil (d+1)(1+1/{\gamma_1})\rceil$ (see Corollary~\ref{c:J>}). In 
Proposition~\ref{l:J_<},  following  Lemmas~22 and   28  from \cite{HB}, we show   that 
 $$
 |J_<^{\gamma_1}|\lappr_{d,{\gamma_1}{ ,m}} L^{d/2+2+\gamma_1(d+1)} \left(\|w\|_{\bar
   N,d+5}+\|w\|_{0,\bar N+3d+4}\right) \,, \quad \bar N= \lceil d^2/{\gamma_1}\rceil-2d\,.
 $$
  
  To analyse $J_0$ we write it as  $J_0=J_0^++J_0^-$ with
\be\lbl{eq:n(0)=}
J^+_0:=\sum_{q>L^{1-\gamma_2}}  q^{-d}
S_q(0) I_q(0)\,,\quad J_0^-:=\sum_{q\leq L^{1-\gamma_2}} q^{-d}
S_q(0) I_q(0)\,,
\ee
for some $0<\gamma_2<1$.
 Lemma~\ref{l:I_A}, which is a combination of Lemmas 16 and 25 from \cite{HB}, modified using the results from Appendix, 
 implies that 
  $$
  \Bigl|J_0^+\Bigr|\lappr_{d,\gamma_2} L^{d/2+2+\gamma_2(d/2-1)}\|w\|_{L_1}\le
  L^{d/2+2+\gamma_2(d/2-1)}\|w\|_{0,d+1}  .
  $$
  
 Finally  Lemma~\ref{l:I_B}, which is a combination of Lemma~13 and  simplified 
 Lemma~31  from \cite{HB} with the  results from Appendix, establishes  that 
 $$
 J_0^- = L^{d} \sigma_\infty(F,w) \sigma(F,m) 
 + O_{d,\gamma_2{ ,m}}\left(
 \left(\|w\|_{d/2-2,d-1} +\|w\|_{0,d+1}\right)L^{d/2+2+\gamma_2}\right)
 $$
 (see \eqref{lead_term} and \eqref{77}). 
 Identity \eqref{eq:N_L=} together with the estimates above  implies 
 the desired result if we choose $\gamma_2=\eps/{(d/2-1)}$ and
 $\gamma_1= \eps/(d+1)$.
 \medskip
 
 \noindent
 {\it Notation.} We write 
 $
 A  \lappr_{a,b}  B
 $ 
 if $A<C B$, where the constant $C$ depends on $a$ and $b$.  Similar, $O_{a,b}( \|w\|_{m_1, m_2}) $ stands for a quantity, bounded in norm by 
 $
 C(a,b) \|w\|_{m_1, m_2}. 
 $
 We denote $e_q(x) = e^{2\pi ix/q}$ and abbreviate $e_1(x)=:e(x)$.

 \section{Components of Singular series  }
 In the presnt section we analyse the sums $S_q(\vc)$ entering in the definition of the  singular series $\sigma(F,m),\,\sigma_p(F,m)$.

 \begin{lemma}\label{l:25HB}({25 in \cite{HB}.})
 	We have $|S_q(\vc)|\lappr_A q^{d/2+1}$, uniformly in $\vc\in\Z^{d}$
 \end{lemma}
\begin{proof} According to \eqref{eq:S}, 
  \begin{equation}\label{eq:S_q=}
  \begin{split}
&|S_q(\vc)|^2 \leq \phi(q) {\sum_{a(\!\!\!\!\!\!\mod
		q)}}^* \, \Big|\sum_{\vb(\!\!\!\!\!\!\mod q)} e_q(a{ F^{mL^2}}(\vb) + \vc\cdot \vb)   \Big|^2  \\
&= \phi(q) {\sum_{a(\!\!\!\!\!\!\mod q)}}^*  \sum_{\vu,\vv(\!\!\!\!\!\!\mod q)} 
e_q\big(a({ F^{mL^2}(\vu)-F^{mL^2}(\vv)}) + \vc\cdot (\vu-\vv)\big),
\end{split}
\end{equation}
where $\phi(q)$ is the Euler totient function. Since
$
F^t(\z) = \tfrac12 A\z\cdot\z -t, 
$
then 
$$
{ F^{mL^2}(\vu)-F^{mL^2}(\vv)}=(A\vv)\cdot \vw
+ F(\vw)=\vv\cdot A\vw + F(\vw). 
$$
So
$$
e_q\big({ F^{mL^2}(\vu)-F^{mL^2}(\vv)}) + \vc\cdot (\vu-\vv)\big)=
e_q\big(aF(\vw) + \vc\cdot \vw\big)\,
e_q(a\vv\cdot A\vw).
$$
Now we see that the summation over $\vv$ in \eqref{eq:S_q=} 
produces  a zero contribution, unless each component of the vector $A\vw$ is divisible by $q$. This property  holds
for $N_\Delta$ possible  values of $\vw$, where $\Delta=\det A$. 
Thus,  
$$|S_q(\vc)|^2 \lappr_A \phi(q) {\sum_{a(\!\!\!\!\!\!\mod
		q)}}^* \, \sum_{\vv(\!\!\!\!\!\!\mod q)} 1 \le
\phi^2(q)\, q^{d}. $$   
\end{proof}
	
We have the following trivial corollary of Lemma~\ref{l:25HB}:
\begin{corollary}\label{c:28HB} 
	We have	$\sum_{q\leq X} |S_q(\vc)|\lappr_A X^{d/2+2}$ uniformly in $\vc\in \Z^{d}$.
\end{corollary}
Recalling the definition \eqref{eq:sigma_p} for a prime $p$ we also have
\begin{lemma}\label{l:31HB}
	For any $d>4$ we have $$\sum_{q\leq X}q^{-d} S_q(0) = \prod_p\sigma_p +
        O_A(X^{-d/2+2}),$$
	where the product is taken over all primes.
\end{lemma}
\begin{proof}
Let us write 
$$
\sum_{q\leq X}q^{-d} S_q(0)=\sum_{q=1}^\infty q^{-d} S_q(0) - \sum_{q\geq X}q^{-d} S_q(0).
$$

By definition
$$
S_{qq'}(\vc) = {\sum_{ a(\!\!\!\!\!\!\mod qq')}}^* \,
\sum_{\vv(\!\!\!\!\!\!\mod qq')} e_{qq'}(aF(\vv) + \vc\cdot
\vv)\,.
$$
When $(q,q') = 1$ we can replace the summation on $a$ by a double
summation on $a_q$ modulo $q$ and $a_{q'}$ modulo $q'$, writing $a=q
a_{q'}+q'a_q$, and the summation on $\vv$ with the double summation
on $\vv_q$ modulo $q$ and $\vv_{q'}$ modulo $q'$, by writing $\vv=
q\bar q \vv_{q'} + q'\bar q' \vv_q$, where $\bar q$ and $\bar q'$ are
defined through $q\bar q=1\,(\!\!\!\mod q')$ and
$q'\bar q' = 1\,(\!\!\!\mod q)$. We substitute in the previous formula
and get
$$
S_{qq'}(\vc) = S_q(\bar q' \vc) S_{q'}(\bar q \vc)\,,
$$
whenever $(q,q')=1$ (cf. Lemma { 23} from \cite{HB}).

This implies the identity 
$$
\sum_{q=1}^\infty q^{-d} S_q(0) = \prod_p\sigma_p.
$$
On the other hand, due to Lemma~\ref{l:25HB}, 
$$
\big|\sum_{q\geq X}q^{-d} S_q(0)\big|\lappr_A \sum_{q\geq X}q^{-d/2+1} \lappr_A X^{-d/2+2}.
$$
\end{proof}

\section{Singular integral}\label{sec:3}

\subsection{Properties of $h(x,y)$}
We construct a function $h(x,y)$ entering Theorem~\ref{th:1},
starting from the weight function
$w_0\in C_0^\infty(\R)$, defined as
\begin{equation}\label{eq:or_weight}
w_0(x) = \left\{\begin{array}{cc}
\exp\left(\frac1{x^2-1}\right) & \mbox{for }|x|<1\\
0 & \mbox{for }|x|\ge 1
\end{array}
\right. \,.
\end{equation}
We denote by
$$
c_0:= \int_{-\infty}^{\infty}w_0(x)\,dx 
$$
and introduce the shifted weight function
\begin{equation*}
  \omega(x)=\frac{4}{c_0}w_0(4x-3)\,.
\end{equation*}
which belongs, of course, to $C^\infty_0(\R)$. Obviously, 
 $ 0\le \omega\le 4e^{-1}/c_0$, $\omega$ is supported on $(1/2,1)$, and
$$
\int_{-\infty}^{\infty}\omega(x)\,dx =1 \,.
$$

The required function $h$ is defined in terms of $\omega$ as
\begin{equation}\label{eq:h}
  \begin{split}
    h(x,y)
& := h_1(x)-h_2(x,y)\,, \mbox{with }\\
h_1(x)&:=\sum_j\frac{1}{xj}\omega(xj)\,,\quad h_2(x,y):= \sum_j\frac{1}{xj}\omega \left(\frac{|y|}{x
  j}\right)  \,.
\end{split}
\end{equation}
For any fixed pair $(x,y)$, each sum on $j$ involved in the definition
contains a finite number of nonzero terms, ranging from
$\tfrac1{2x}$ to $1/x$ for the summation in $h_1$, and from $|y|/x$ to
$2|y|/x$ for $h_2$.  So $h$ is a smooth function. 

In \cite{HB}, Section~3, it is shown how to derive Theorem~\ref{th:1} from
the definition \eqref{eq:h}.\footnote{Actually, it is proven how the
  function $h$ defined through \eqref{eq:h} can provide a
  representation of $\delta(n)$ for any weight function $\omega\in
  C_0^\infty(\R)$ supported on $[1/2,1]$.} Here we limit ourselves
 to providing some relevant properties of $h$,  
proved in Section~4 of \cite{HB}. In particular these properties imply that for small $x$,
  $h(x,y)$ behaves as the Dirac delta function in $y$

\begin{lemma}[Lemma 4 in \cite{HB}]\label{l:4}
  We have:
  \begin{enumerate}
\item    $h(x,y)=0$ for $x\ge 1$ and $|y|\le x/2$.
\item If $x\le 1$ and $|y|\le x/2$, then  $h(x,y)=h_1(x)$ and for any $m\geq 0$
  $$ 
  \frac{\partial^m h(x,y)}{\partial x^m} \lappr_m \frac{1}{x^{m+1}}\,.
  $$
  \item If $|y|\ge x/2$, then for any $m,n\geq 0$
  $$
  \frac{\partial^{m+n} h(x,y)}{\partial x^m\partial y^n} \lappr_{m,n}
  \frac{1}{x^{m+1}|y|^n}\,.
  $$
  \end{enumerate}
  \end{lemma}
Lemma~\ref{l:4} immediately implies
\begin{corollary}\lbl{c:h}
	For any $x,y\in\R_>\times\R$ we have $|h(x,y)|\lappr 1/x$.
\end{corollary}

\begin{lemma}[Lemma 5 in \cite{HB}]\label{l:5}
  Let $m,n,N\ge 0$. Then
  $$
  \frac{\partial^{m+n} h(x,y)}{\partial x^m\partial y^n} \lappr_{N,m,n}
  \frac{1}{x^{1+m+n}}\left(\delta(n)x^N + \min\left\{1,\left(x/|y|\right)^N\right\} \right)\,.
  $$
  \end{lemma}

\begin{lemma}[Lemma 6 in \cite{HB}]\label{l:6}
  Fix $ X\in \R_{>0}$ and let $x<C\min\left\{1,X\right\}$ for $C>0$. Then for any $N\ge0$, 
  $$
 \int_{-X}^X h(x,y)\,dy = 1 + O_{N,C}\left(Xx^{N-1}\right) +
 O_{N,C}\left(\frac{x^N}{X^N} \right)\,.
  $$
  \end{lemma}

\begin{lemma}[Lemma 8 in \cite{HB}]\label{l:8}
  Fix $X\in \R_{>0}$ and $n\in \N$. Let $x<C\min\left\{1, X\right\} $ for
  $C>0$. Then 
  $$
 \int_{-X}^X y^nh(x,y)\,dy \lappr_{N,C} X^n\left(Xx^{N-1} +
\frac{x^N}{X^N} \right)\,.
  $$
\end{lemma}

The previous results are used to prove the key Lemma~9 of \cite{HB},
which can be extended to the following
\begin{lemma}[{Lemma 9 of \cite{HB}}]\label{l:9} Let an
  integrable function $f$ be $C^{M-1}$-smooth, $M\geq 1$, and be such that
  $f^{(M-1)}$ is absolutely continuous in $[-1,1]$. Let  $x\le C$ with
  $C>0$, then, for any ${\gamma_1}>0$, 
  \begin{equation*}
\int_\R f(y) h(x,y) \, dy = f(0) + O_{M,C,{\gamma_1}}\left(x^{M-{\gamma_1}}
\left(\frac{1}{X}\int_{-X}^{X}|f^{(M)}(y)|\,dy+\|f\|_{L_1}\right)\right)\,,
  \end{equation*}
  where $X:=\min\left\{1,x^{1-\gamma_1/(M+1)}\right\}$
\end{lemma}
    {\it Proof.} By
    Lemma~\ref{l:5},  
    $h(x,y) \lappr_M x^M$ if $|y|\ge X$, so that the integral on the tails can be
    bounded by
    \begin{equation}\label{eq:l9-1}
    \int_{|y|\ge X} f(y) h(x,y) \, dy \lappr_M x^M \int_\R |f(y)|\lappr_M
    x^M \|f\|_{L_1}\,.
    \end{equation}

    For the integral for $|y|<X$, instead, we take the Taylor
    expansion of $f(y)$ around zero and get
    \begin{equation}\label{eq:l9-2}
    \begin{split}
        \int_{-X}^X& f(y) h(x,y) \, dy  \\&
        = \sum_{j=0}^{M-1}
        \frac{f^{(j) }(0)}{j!}\int_{-X}^X y^j h(x,y)\, dy +
        O_M\left(\frac{X^{M}}{x} \int_{-X}^{X}|f^{(M)}(y)|\,dy\right)\,,
        \end{split}
    \end{equation}
    by Corollary~\ref{c:h}. Then, we fix $N$ and use Lemma~\ref{l:6} to
    get
    \begin{equation}\label{eq:l9-3}
        f(0) \int_{-X}^X  h(x,y)\, dy = f(0) + 
         O_{N,C}\left(\|f\|_{0,0}\left(Xx^{N-1} 
        +\frac{x^N}{X^N}\right)\right),
    \end{equation}
    while, from Lemma~\ref{l:8}, for any $j>0$ we have
    \begin{equation}\label{eq:l9-4}
\left|      \frac{f^{(j) }(0)}{j!}\int_{-X}^X y^j h(x,y)\, dy\right|
   \lappr_{N,j,C}   \|f\|_{j,0}X^j\left(Xx^{N-1} 
   +\frac{x^N}{X^N}\right)\,.
    \end{equation}
    Putting together \eqref{eq:l9-1}--\eqref{eq:l9-4}, putting $N=M(M+1)/\gamma_1$
    for $x\le 1$,  $N=M$ for $x>1$, 
    and using the definition of $X$ we obtain the proof.
    \qed

    \subsection{The approximation for $I_q(0)$}
We have the following proposition, which replaces Lemmas~11, 13 and
Theorem~3 of \cite{HB},  not assuming  that  $0\notin\supp w$:

    \begin{proposition}\lbl{l:I_q(0)=}
   Let $F$ be the quadric $F_0$, see \eqref{F0}, (so $d$ is an even
   number). Let  $q\le C L$, with $C> 0$. Then for any $1\le M<
   d/2-1$
    \be\lbl{I_q(0)=}
      I_q(0) = L^{d} \sigma_\infty(F,w{ , m})      +
      O_{M,C,{\gamma_1}{ ,d, m}}\left(q^ML^{d-M+{\gamma_1}}\|
      w\|_{M,d-1} + \|
      w\|_{0,d+1}\right) \,.
      \ee
{       If $1\le M=d/2-1$, { $m\neq 0$ and $C\le m/2$, } instead, we have}
      \begin{equation*}
      I_q(0) = L^{d} \sigma_\infty(F,w{ , m})      +
      { |\log m|}O_{C,{\gamma_1}{ ,d, m}}\left(q^ML^{d-M+{\gamma_1}}\|
      w\|_{M,d-1} + \|
      w\|_{0,d+1}\right) \,.
      \end{equation*}
      Finally, if $1\le M=d/2-1$ and $m=0$, we have
      \begin{equation*}
      \begin{split}
      I_q(0)& = L^{d} \sigma_\infty(F,w{ , m})  \\    
      &+
      O_{C,{\gamma_1}{ ,d, m}}\left(q^ML^{d-M+{\gamma_1}}\|
      w\|_{M,d-1}|\log(q/L)| + \|
      w\|_{0,d+1}\right) \,.
      \end{split}
      \end{equation*}
      
    \end{proposition}
{\it Proof.} Let us  write $I_q(\vc)$ as 
\be\lbl{Iq-t_Iq}
I_q(\vc)=L^{d} \tilde I_q(\vc),
\ee
where
\be\lbl{tilde_I_q}
\tilde I_q (\vc) = \int_{\R^{d}} w(\z)\,
h\left(\frac qL, { 
F^m}(\z)\right) e_q(-\z\cdot \vc L)
\,d\z\,.
\ee
Applying the  co-area formula (see e.g. \cite{Chav}, p.138)  to the integral in \eqref{tilde_I_q}
with $\vc=0$ we get that  
$$
\tilde I_q(0)=\int_\R \II({ m+}t) h(q/L,t)\, dt\,,
$$
where
$$
\II(t) = \int_{\Sigma_t} w(\z)\,\mu^{\Sigma_t}(d\z)\,
$$
(the measure $\mu^{\Sigma_t}$ is the same as in \eqref{lead_term}). 
   Since $q\le CL$, then on 
account of Lemma~\ref{l:9},
$$
\int_\R \II({  m+}t) h(x,t)\,dt = \II({  m}) +
O_{M,C}\Big(x^{M-{\gamma_1}} 
\big(\frac1X\int_{-X}^{X}|\II^{(M)}({  m+}t)|\,dt
+\|\II\|_{0,2}\big)\Big),
$$
where $X:=\min\{1,x^{1-\gamma_1/(M+1)}\}$.
In order to conclude the proof, we make use of Proposition~\ref{p_2}
of Appendix~\ref{sec:I(t)}, which guarantees that
$$
\int_{-X}^{X}|\II^{(M)}({  m +}t)|\,dt{ \lappr_{d,m}} \left\{\begin{array}{cc}
X\|w\|_{M,d-1}& M< d/2-1\,, \\
{  X |\log m|\|w\|_{M,d-1}}&{  M=d/2-1, m\neq 0, C\le m/2,}\\
X|\log X| \|w\|_{M,d-1} &   M = d/2-1\,,
\end{array}
\right.
$$
and that  $\|\II\|_{L_1}  \lappr_{d} \|w\|_{0,d+1}$.

\section{The $J_0$ term }\label{sec:4}

In this section we prove the following proposition concerning  the
term $J_0$ from \eqref{eq:N_L=} when $F=F_0$, not assuming that $0\notin\supp w$:

\begin{proposition}\lbl{l:n(0)}
 Let $F$ be the quadric $F_0$. 
  Assume that $w\in C^{d/2-2}(\R^{d})$. Then, for any $0<\gamma_2<1$,
	$$
	\big|J_0-L^{d}\sigma_\infty(w; F, { m})\prod_p\sigma_p\big|
        \lappr_{\gamma_2{ ,d,m}}
	 L^{\frac{d}2+2 +\gamma_2(\frac{d}2-1)}
        \left(\|w\|_{d/2-2,d-1}+\|w\|_{0,d+1}\right).
	$$
\end{proposition}
{\it Proof.}
To establish Proposition~\ref{l:n(0)} we write $J_0$ in the form
\eqref{eq:n(0)=}. 
Then the assertion   follows from Lemmas~\ref{l:I_A} and
\ref{l:I_B} below,  estimating $J_0^+$ and $J_0^-$
separately, and noting then that $\|w\|_{L_1}\le
\|w\|_{0,d+1}$  for $d\ge 3$.
\qed

\begin{lemma}\lbl{l:I_A}
	Assume that $w\in L_1(\R^{d})$ and $d>2$. Then we have the
        bound $|J_0^+|\lappr_{A}
        L^{d/2+2 +\gamma_2(d/2 -1)}\|w\|_{L_1}$.
\end{lemma}
{\it Proof.}
Since according to Lemma~\ref{l:25HB} $|S_q(0)|\lappr_A q^{d/2+1}$, then 
$$
|J_0^+|\lappr_A\sum_{q>L^{1-\gamma_2}} q^{-d/2+1}I_q(0).
$$
Dividing in the definition \eqref{eq:I'} of the integral $I_q$ the variable of integration $\z$ by $L$, we get
$$
I_q(0)=L^{d}\int_{\R^{d}} w(\z) h(q/L, { F^m}(\z))\,d\z\,.
$$
Corollary~\ref{c:h} implies that
$
\ds{|I_q(0)|\lappr \fr{L^{d+1}}{q}  \|w\|_{L_1}.} 
$
Therefore,
\begin{equation*}
  \begin{split}
|J_0^+|&\lappr_A L^{d+1}\|w\|_{L_1} \sum_{q>L^{1-{\gamma_2}}} q^{-d/2 }
\lappr_A  L^{d+1}\|w\|_{L_1}L^{(-d/2+1)(1-\gamma_2)}\\
&=
L^{d/2+2 +\gamma_2(d/2-1)}\|w\|_{L_1}.
  \end{split}
\end{equation*}
\qed

\smallskip

\begin{lemma}\lbl{l:I_B}  Let $F$ be the quadric $F_0$. 
	Assume that $w\in C^{d/2-2}(\R^{d})$. Then 
	$$
	J_0^-=L^{d}\sigma_\infty(w; F, { m})\prod_p\sigma_p + O_{\gamma_2{ ,d,m}}\big[ \left(\|w\|_{d/2-2,d-1}+\|w\|_{0,d+1}\right) L^{d/2+2 +\gamma_2} \big].
	$$
\end{lemma}
{\it Proof.}
Inserting \eqref{I_q(0)=}  into the
definition of the term $J_0^-$, we get 
$
J_0^-=I_A+I_B,
$
where
\begin{align}\non
&I_A :=L^{d}\sigma_\infty(w; F, { m})\sum_{q\leq L^{1-\gamma_2}} q^{-d}S_q(0), 
\\\non 
&|I_B|\lappr_{M,\delta{ ,d,m}}L^{d-M+\delta}\left(\|w\|_{M,d-1}+\|w\|_{0,d+1}\right)
\sum_{q\leq 
  L^{1-{\gamma_2}}}S_q(0) q^{-d+M} \,,
\end{align}
for $M\le d/2-2$ and any $\delta>0$. \ 
Lemma~\ref{l:31HB} implies that $$\sum_{q\leq L^{1-\gamma_2}} q^{-d}S_q(0)=\prod_p\sigma_p + O(L^{(-d/2+2)(1-\gamma_2)}),$$ so
$$
I_A=L^{d}\sigma_\infty(w; F, { m})\prod_p\sigma_p + O(\sigma_\infty(w; F, { m})
L^{d/2+2 +\gamma_2(d/2-2)})\,,
$$
whereas $|\sigma_\infty(w; F, { m})|=|\II({ m})|\le
\|w\|_{0,d-1}$  on account of Proposition~\ref{p_2}.  As for the term
$I_B$, Lemma~\ref{l:25HB} implies that
$$
|I_B|\lappr_{M,\delta{ ,d,m}}
L^{d-M+\delta}\left(\|w\|_{M,d-1}+\|w\|_{0,d+1}\right) \sum_{q\leq 
  L^{1-\gamma_2}}q^{-d/2+1+M}\,. 
$$
Choosing $M=d/2-2$ and $\delta =\gamma_2/2$, we get 
\begin{equation*}
  \begin{split}
|I_B|&\lappr_{\delta{ ,d,m}} \left(\|w\|_{d/2-2,d-1}+\|w\|_{0,d+1}\right)
L^{d/2+2+\delta}\ln L\\
&\lappr_{\gamma_2{ ,d,m}}
\left(\|w\|_{d/2-2,d-1}+\|w\|_{0,d+1}\right) \, L^{d/2+2 +\gamma_2}\,.
  \end{split}
  \end{equation*}
\qed

\section{The  $J_>^{\gamma_1}$ term}\label{sec:5}

We provide here an estimate of the term $J_>^{\gamma_1}$ defined in
\eqref{eq:J><}.  The key point of the proof 
 is an adaptation of Lemma~19 of
\cite{HB} to our case; we recall the notation~\eqref{Iq-t_Iq}. 

\begin{proposition}\label{l:19}
   For any $ N>0$ and $w\in C^N(\R^d)$
\be\lbl{eq:19}
|\tilde I_q(\vc)|\lappr_{d,N,A{ ,m}} \frac Lq |\vc|^{-N}\left\|w\right\|_{N,2N+d+1}
  \ee
\end{proposition}

{\it Proof.}
We call $f_q(\z):= w\left(\z\right)
  h\left(\frac qL, 
 {  F^m}(\z)\right)$. Since
  $$
  \frac{i}{2\pi}\frac{q}{L} |\vc|^{-2}\left( \vc\cdot \nabla_{\z}\right)
  e_q(-\z\cdot \vc L)  =  e_q(-\z\cdot \vc L)\,,
  $$
then integrating by parts $N$ times \eqref{tilde_I_q} we get that

\begin{equation*}
  \begin{split}
    \left| \tilde I_q(\vc)\right| &\le \left(\frac{q}{2\pi L} |\vc|^{-2}\right)^N \int_{\R^{d}} \left|\left( \vc\cdot
    \nabla_{\z}\right)^N f_q(\z)\right|\,d\z\\
    &\lappr_{d,N,A} \left(\frac qL\right)^N |\vc|^{-N}
    \sum_{0\le n\le N}\int_{\R^{d}}\max_{0\le l\le n/2}
    \left|\frac{\partial^{n-l}}{\partial y^{n-l}} h\left(\frac
    qL, {  F^m}(\z)\right) \right|\\
        &\qquad\qquad\times|\z|^{n-2l}
    \left|
    \nabla_{\z}^{N-n} w(\z)\right|\,d\z
\,,
  \end{split}
\end{equation*}
where $\ds{\frac{\partial}{\partial y} h}$ stands for the derivative of  $h$ with respect to 
 the second argument. 

Let us distinguish then
 two cases. When $q\le L$, Lemma~\ref{l:5} (with $N=0$) 
 implies that
\begin{equation*}
  \begin{split}
\max_{0\le l\le n/2}\left|\frac{\partial^{n-l}}{\partial y^{n-l}} h\left(\frac
    qL, {  F^m}(\z)\right) \right| |\z|^{n-2l}
    \left|    \nabla_{\z}^{N-n} w(\z)\right|\le\\ \left(
    L/q\right)^{n+1}\lan \z\ran^{-d-1} \|w\|_{N-n,n+d+1}\,,
  \end{split}
\end{equation*}
from which \eqref{eq:19}  follows since $n\le N$. 
  When $q>L$, because of
Lemma~\ref{l:4}, point 1, $h$ is different from zero only if
\be\lbl{h ne 0}
2| {  F^m}(\z)| >\frac qL .
\ee  
\newpage
Then for such $\z$ and for $l\le n$,  point 3 of Lemma~\ref{l:4} implies that

$$
\left|\frac{\partial^{n-l}}{\partial y^{n-l}} h\left(\frac qL,  {  F^m}(\z)\right)\right| \lappr_{n-l} \frac Lq\frac{1}{| {  F^m}(\z)|^{n-l}}\lappr_{n-l}\Big(\frac{L}{q}\Big)^{n-l+1}.
$$
So 
\begin{equation*}
\begin{split}
\max_{0\le l\le n/2} \left|\frac{\partial^{n-l}}{\partial y^{n-l}} h\left(\frac
qL, {  F^m}(\z)\right) \right| |\z|^{n-2l}
\left|
\nabla_{\z}^{N-n} w(\z)\right|\le\\ \max_{0\le l\le n}\frac{\left(
	L/q\right)^{n-l+1}}{\lan \z\ran^{2(N-n+l)}}
\frac{\|w\|_{N-n,2N-n+d+1}}{\lan\z\ran^{d+1}}\,.
\end{split}
\end{equation*}
Since from  \eqref{h ne 0} we have that $q/L \lappr_{A{, m}}
\lan\z\ran^2$,  
then the first fraction above is bounded by $(L/q)^{N+1}$,  and again 
 \eqref{eq:19} follows. 
\qed

As a corollary, we can infer the desired estimate for $J_>^{\gamma_1}$:
\begin{corollary}\label{c:J>}
  For $J_>^{\gamma_1}$ defined in \eqref{eq:J><} and $d>2$ we have
  $$
|J_>^{\gamma_1}|\lappr_{d,{\gamma_1},A{ , m}}\|w\|_{N_0,2N_0+d+1}\,,
$$
where $N_0:=\lceil (d+1)(1+1/{\gamma_1})\rceil$.
\end{corollary}

{\it Proof.}
By the definition of $J_>^{\gamma_1}$ we have
\begin{equation*}
  \begin{split}
|J_>^{\gamma_1}|&\lappr_d \sum_{s\ge L^{\gamma_1}} s^{d-1}\sum_{q=1}^\infty
q^{-d}\sup_{\|\vc\|_1=s}\ |S_q(\vc)| |I_q(\vc)|\\
&\lappr_{d,A} \sum_{s\ge L^{\gamma_1}} s^{d-1}\sum_{q=1}^\infty
q^{1-d/2}L^d\sup_{\|\vc\|_1=s} |\tilde I_q(\vc)|\\
&\lappr_{d,N,A{ , m}} \sum_{s\ge L^{\gamma_1}} s^{d-1}\sum_{q=1}^\infty
q^{-d/2} s^{-N} L^{d+1}\|w\|_{N,2N+d+1}\,,
  \end{split}
\end{equation*}
  where the second line follows through Lemma~\ref{l:25HB}, while the
  third one via Proposition~\ref{l:19}. We choose $N=\lceil
  (d+1)(1+1/{\gamma_1})\rceil$ and get that
  $$
  \sum_{s\ge L^{\gamma_1}} s^{d-1} s^{-N} L^{d+1}\le
  \sum_{s\ge L^{\gamma_1}} s^{-2} \lappr 1 \,,
  $$
  while the sum in  $q$ is bounded, too. This concludes the proof.
  \qed
  

 \section{ The $J^{\gamma_1}_<$ term }\label{sec:6}
 \subsection{The estimate}\label{subsec:6.1}
 Our next (and final) goal is to estimate the term $J^{\gamma_1}_<$ from \eqref{eq:N_L=}.
 
 \begin{proposition}\lbl{l:J_<} Let $w\in C^{\bar N}(\R^d)$, where
   $\bar N=\bar N(d,{\gamma_1}):= \lceil d^2/{\gamma_1}\rceil-2d$, and
   $0<\gamma_1< d/2-1$. Then, 
 	$$|J_<^{\gamma_1}|\lappr_{A,d,{\gamma_1}{ ,m}}  L^{d/2+2+{\gamma_1}(d+1)} \left(\|w\|_{\bar
 		N,d+5}+\|w\|_{0,\bar N + 3d+4}\right) \,.$$ 
 \end{proposition}
 Proposition~\ref{l:J_<} follows from the next lemma   which is a
 modification of Lemma~22 in \cite{HB}:
 \begin{lemma}\lbl{l:22}
 	For $|\vc|\leq L^{\gamma_1},$ $\vc\ne 0$,
 	$$|I_q(\vc)|\lappr_{A,d,{\gamma_1}{ ,m}}
 	L^{d/2+1+{\gamma_1}}q^{d/2-1} \left(\|w\|_{\bar N,d+5}+\|w\|_{0,\bar N
 		+3d+4}\right)\,,
 	$$
 	where $\bar N$ and $\gamma_1$ are the same as above.
 \end{lemma}
 
 {\it Proof of Proposition \ref{l:J_<}}.
 Accordingly to Lemma~\ref{l:25HB},
 \begin{align*}
 |J^{\gamma_1}_<| &\lappr_A \!
 \sum_{\vc\ne 0,\,|\vc|\leq
 	L^{{\gamma_1}}} \sum_{q=1}^\infty q^{-d}
 q^{d/2+1} |I_q(\vc)|\lappr_{A,d} L^{d{\gamma_1}}\!\! \max_{\vc\ne0:\,|\vc|\leq L^{\gamma_1}} |I_q(\vc)|\,
 \sum_{q=1}^\infty q^{-d/2+1}
 \\
 &=L^{d{\gamma_1}} \big(\sum_{q<L} + \sum_{q\geq L}\big) q^{-d/2+1}
 \max_{\vc\ne0:\,|\vc|\leq L^{\gamma_1}} |I_q(\vc)|=:J_{-} + J_{+}.
 \end{align*}
 Corollary~\ref{c:h} together with \eqref{Iq-t_Iq}, \eqref{tilde_I_q} implies
 \be\lbl{Iq_large_q}
 |I_q(\vc)|\lappr \frac{ L^{d+1}}q \|w\|_{L_1}\,,
 \ee
 so that
 $$
 J_{+}\lappr L^{d{\gamma_1}}  L^{d+1} \|w\|_{L_1} \sum_{q\geq L} q^{-d/2}  \lappr  L^{d{\gamma_1} + d/2+2} \|w\|_{L_1}.
 $$
 By Lemma~\ref{l:22} we get
 \begin{equation*}
 \begin{split}
 J_{-}
 &\lappr_{A,d,{\gamma_1}{ ,m}}  L^{d{\gamma_1}} \left(\|w\|_{\bar N,d+5}+\|w\|_{0,\bar N
 	+3d+4}\right)\sum_{q< L}  L^{d/2+1+{\gamma_1}}
 \\
 &\le   \left(\|w\|_{\bar N,d+5}+\|w\|_{0,\bar N
 	+3d+4}\right) L^{{\gamma_1}(d+1)+d/2+2}\,.
 \end{split}
 \end{equation*}
 \qed

\subsection{Proof of Lemma~\ref{l:22}}
We begin with
\subsubsection{An application of the inverse Fourier transform}
 Note that the proof is nontrivial only for $q\lappr L$ since  for
 any $\alpha>0$, the bound \eqref{Iq_large_q} implies
 that  for $q\ge \alpha L$ we have
 $$|I_q(\vc)| \lappr_\alpha L^d\|w\|_{L_1} \lappr_{\alpha,d}  L^{d/2+1}q^{d/2-1} \|w\|_{L_1}\,.$$
 
 Let us take a small enough $\alpha=\alpha(d,\gamma_1, A) \in(0,1)$   and assume that $q< \alpha L$. Consider the positive
 function $w_2(x)=1/(1+x^2)$ and set 
 $$
 \tilde w(\z):= \frac{w(\z)}{w_2( {  F^m}(\z))} = {w(\z)} (1+  {  F^m}(\z)^2). 
 $$ 
 
 Let
 \be\lbl{p(t)=}
 p(t):= \int_{-\infty}^{+\infty}  w_2(v) h(q/L,v)  e(-tv) \,dv, \quad e(x):=e_1(x)=
 e^{2\pi i x}.
 \ee
 This is  the Fourier transform of the function $w_2(\cdot)h(q/L,\cdot)$.
 Then, expressing $w_2 h$ via $p$ 
 by the inverse Fourier transform, we find that 
 $$
 w(\z)h(q/L, {  F^m}(\z))= \tilde w(\z)\int_{-\infty}^{+\infty} p(t)e(t {  F^m}(\z))\,dt.
 $$
 Inserting this representation into \eqref{tilde_I_q} we get
 \begin{equation*}
 \tilde I_q(\vc)=\int_{-\infty}^{+\infty} dt\, p(t){ e(-tm)}\,\int_{\R^{d}}d\z\, \tilde w(\z) e\big(tF(\z)-\vu\cdot \z\big),
 \end{equation*}
 where we define
 $$
 \vu:=\vc\, L/q.
 $$
 Note that $|\vu|\geq \alpha^{-1}>1$. Now let us 
   denote
 $
 W_0(x) = c_0^{-d}\prod_{i=1}^{d}w_0(x_i) 
 $
 (see \eqref{eq:or_weight}). Then $W_0\in C_0^\infty (\R^d)$, $W_0\ge0$ and 
 \be\label{then}
  \supp W_0 =[-1, 1]^d \subset \{ |x| \le \sqrt{d} \},\quad \int W_0 =1.
 \ee
 Let us set
 $
 \de  = | \vu_0|^{-1/2} 
 $
 and write $\tilde w$ as 
 $$
 \tilde w(\z) = \de^{-d} \int W_0\left( \frac{\z-\va}{\de}\right) \tilde w(z)\,d\va. 
 $$
 Then  setting $\vb:=\displaystyle{\frac{\z-\va}{\de}}$  we get that 
 $$
 |\tilde I_q(\vc)|\leq \int_{\R^{d}}d\va\,\int_{-\infty}^{+\infty} dt\, |p(t)| 
 |I_{\va,t}|,
 $$
 where 
 \begin{equation*}
 I_{\va,t}:=\int_{\{ |\vb| \le \sqrt{d} \} }  W_0(\vb) \tilde w(\z) \, e(tF(\z) - \vu\cdot\z)\,d\vb, \qquad \z:=\va+\de\vb
 \end{equation*}
 (in virtue of \eqref{then}). 
 Consider the exponent in the integral $ I_{\va,t}$:
 $$
 f(\vb)=f_{\va,t}(\vb):=tF(\va+ \de\vb) - \vu\cdot (\va+ \de\vb).
 $$
 At  the next step we will estimate  integral $I_{\va,t}$, 
 regarding  $(\va,t)$ as a parameter, and distinguishing two cases:

 1. $(\va,t)$ belongs to the   "good" domain $S_R$, where 
 $$
 S_R = \big\{ ( \va,t): 
  |\nabla f(0)|=\de |t A\va - \vu| \ge 
  R\big\lan  t /  |\vu \big\ran =R\lan \de^2t\ran \big\}. 
 $$ 
 
 2. $(\va,t)$ belongs to the  "bad" set $S_R^c = (\R^d \times \R) \setminus S_R$. 
 \smallskip
 
 \noindent 
 Here 
 $$\ds{1\leq  R\leq |\vu|^{1/3}}$$
 is a parameter to be chosen later.

 \subsubsection{Integrals over $S_R$ and $S_R^c$.}
 We consider first the integral over the good set~$S_R$: 
 \begin{lemma}\lbl{l:good_set}
 	For any $N\ge 0$ and $R\ge 2\|A\|\sqrt{d}$ we have
 	\be\lbl{int_to_bound}
 	\int_{S_R}d\va\, dt |p(t)| \, |I_{\va, t} | 
 	\lappr_{d,N{ , m}} \frac{L}{q}R^{-N}\|w\|_{N,d+5}\,.
 	\ee
 \end{lemma}
 
 {\it Proof.}
  Let $\vl:=\nabla f(0)/|\nabla
 f(0)|$  and  $\cL =\vl\cdot \nabla_{\vb}$.  Then for $(\va,t)\in S_R$, 
 \be\lbl{nabla_f}|\cL f(\vb)|\ge
 |\nabla f(0)| - \delta^2|t||A\vb|\ge  R\lan\de^2 t \ran
 - \de^2|t|\|A\| \frac{R}{2\|A\|} \ge \tfrac12 R\lan\de^2 t \ran
 \geq R/2.
 \ee
 Since 
 $
 (2\pi i \cL f(b))^{-1} \cL e(f(b)) = e(f(b)), 
 $
 then integrating by parts  $N$ times in the integral for $ I_{\va,t}$   we get that 
 $$
 |I_{\va,t}| \lappr_{d,N}
 \sup_{|b_i|\le
 	1\, \forall i}\,\max_{0\leq k\leq N}\left|\cL^{N-k} \tilde
 w(\delta \vb +\va)\frac{\big(\cL^2f(\vb)\big)^k}{\big(\cL f(\vb)\big)^{N+k}}\right| \,,
 $$
 where we have used that $\cL^m f(\vb)=0$ for every $m\geq 3$. 
 Since $|\cL^2f(\vb)|\leq \de^2|t||\vl\cdot A\vl|\leq\de^2|t|\|A\|$, then  in view of \eqref{nabla_f} 
 $$
 \left|\frac{\cL^2 f(\vb)}{\cL f(\vb)}\right|  \leq\frac{\de^2|t|\|A\|}{\tfrac12 R\lan\de^2 t \ran }=\frac{2\|A\|}{R}\leq \frac{1}{\sqrt{d}}.
 $$
 So using that $\ds{\Big|\frac{1}{\cL f(\vb)}\Big|\leq\frac2R}$ by \eqref{nabla_f}, we find
 $$
 |I_{\va,t}| \lappr_{d,N} R^{-N}
 \sup_{|b_i|\leq 1 \,\forall i}\,\max_{0\leq k\leq N}\left|\cL^{k} \tilde
 w(\delta \vb +\va)\right|. 
 $$
 Thus, denoting by  ${\bf 1}_{S_R}$  the indicator function of the set $S_R$, we have 
 \be\non
 \begin{split}
 	\int_{\R^{d}}|I_{\va,t}|{\bf 1}_{S_R}\,d\va
 	&\lappr_{d,N} R^{-N}\!\int_{\R^{d}} \frac{d\va}{\lan\va\ran^{d+1}} 
 	\Big(\lan\va\ran^{d+1} \!\!
 	\sup_{|b_i|\le 1\,\forall i}\, \max_{0\leq k\leq N}\big|\cL^k\tilde
 	w(\delta \vb +\va)\big| \Big)
 	\\
 	& \lappr_{d,N} R^{-N}\|\tilde w\|_{N,d+1}
        \lappr_{d,N{ , m}} R^{-N}\|w\|_{N,d+5}  \,,
 \end{split}
 \ee
 for every $t$.  Then 
 \be
 \mbox{l.h.s. of \eqref{int_to_bound}}\lappr_{d,N{ , m}} R^{-N}\|w\|_{N,d+5} \int_{-\infty}^{+\infty} |p(t)| \, dt.
 \ee
 
 It remains to show that
 \begin{equation}\label{intp(t)}
 \int_{-\infty}^\infty|p(t)|dt\lappr L/q\,.
 \end{equation}
 In virtue of Lemma~\ref{l:5},
 $$
 \Big|\frac{\p^k}{\p v^k} h(x,v)\Big| \lappr_k x^{-k-1}\min\{1,x^2/v^2\}\,,\quad k\geq 1,
 $$
 and by Corollary~\ref{c:h}, $|h(x,v)|\lappr x^{-1}$.
 Then, a simple integration by parts in \eqref{p(t)=} shows that,
 for any $M\geq 0$, 
 \[\begin{split}
 |p(t)| & \lappr_M |t^{-M}|
 \Big( 
 \int_{-\infty}^\infty
 |w_2^{(M)}(v)|x^{-1}\,dv \\
 &+  \max_{1\leq k \leq M}\int_{-\infty}^\infty
 |w_2^{(M-k)}(v)| \, x^{-k-1}\min\big\{1,\frac{x^2}{v^2}\big\}\,dv\Big),
 \end{split}
 \] 
 where $x:=q/L<\alpha$. Writing  the latter integral as a sum $\int_{|v|\leq x} + \int_{|v|>x}$ we see that
 $$
 \int_{|v|\leq x}= x^{-k-1}\int_{|v|\leq x} |w_2^{(M-k)}(v)|\,dv \lappr_k x^{-k}
 $$
 and 
 $$
 \int_{|v|>x} = x^{-k+1}\int_{|v|>x} \frac{ |w_2^{(M-k)}(v)|}{v^2}\,dv \lappr_k x^{-k}.
 $$
 Then, since $x=q/L<1$,
 \begin{equation}\lbl{p(t)<}
 |p(t)| \lappr_M \left(\frac{q}{L}|t|\right)^{-M}\,,\qu M\geq 0.
 \end{equation}
 Choosing $M=2$ when $|t|>L/q$ and $M=0$ when $|t|\leq L/q$ we get \eqref{intp(t)}.

 \qed
 
 Then we study the integral over the bad set ${S_R}^c$.
 \begin{lemma}\label{l:bad_set}
 	For any  $1\le R\le |\vu|^{1/3}$ and $0<\beta<d^2/2$ we have
 	$$
 	\int_{S_R^c}d\va \,dt |p(t)| |I_{\va, t}| 
 	\lappr_{A,d{ ,m}}  R^{d}|\vu|^{-d/2+1+\beta}\|w\|_{0,K(d,\beta)}\,,
 	$$
 	where $K(d,\beta)=  d+\lceil d^2/2\beta\rceil+4$.
 \end{lemma}
 
 {\it Proof.}
 On ${S_R}^c$ we use for $I_{\va, t}$ the easy upper bound
 \begin{equation}\lbl{I(a,t)<}
 |I_{\va, t}|\lappr_d \sup_{|b_i|\le 1\,\forall i}|\tilde w(\de \vb +\va)|\ \leq \|\tilde w\|_{0,0}.  
 \end{equation}
 The fact that $(\va,t)\in {S_R}^c$ implies that the integration in 
  $d \va$ for a  fixed $t$ is restricted to the region where
$$
 \left|A\va -\frac{\vu}{t} \right| \le  \frac{R}{\de |t|} \lan  t / |\vu| \ran \,,
 $$
 or
  \begin{equation}\label{bound_int_a}
 \left|\va -\frac{A^{-1}\vu}{t}\right|\le  \|A^{-1}  \| \, \frac{R}{\de |t|} \, \lan  t / |\vu| \ran \,.
 \end{equation}
 We first consider  the case $|t|\ge |\vu|^{1-\beta/d}$. Since $\de=|\vu|^{-1/2}$ and $|\vu| >1$, then 
 \be\label{grr}
 \frac{R}{\de
 	|t|}\lan  t / |\vu| \ran \le R\max(|\vu|^{-1/2+\beta/d}, |\vu|^{-1/2})=R|\vu|^{-1/2+\beta/d}\,.
\ee
 In view of  \eqref{I(a,t)<} -\eqref{grr},
 $$
 \left|\int_{\R^{d}} |I_{\va,t}|{\bf 1}_{{S_R}^c}(\va,t)d\va\right|\lappr_{A,d}
 {R^{d}}|\vu|^{-d/2+\beta} \|\tilde
 w\|_{0,0}\lappr_{A,d{ , m}} {R^{d}}|\vu|^{-d/2+\beta} \|  w\|_{0,4} \,. 
 $$
 Since by  \eqref{intp(t)} 
 $
 \int_{|t|\ge |\vu|^{1-\beta/d}} |p(t)|\,dt
  \lappr \frac{L}{q}\le|\vu|\,,
 $
 then 
 \begin{equation}\label{t>u}
 \int_{|t|\ge |\vu|^{1-\beta/d}}dt \,\int_{\R^{d}}d\va \, |p(t)| |I_{\va, t}|{\bf 1}_{{S_R}^c}(\va,t) \lappr_{A,d{ ,m}}
 R^{d}|\vu|^{-d/2+1+\beta} \|  w\|_{0,4}\,.
 \end{equation}
 
 Now let  $|t| \le |\vu|^{1-\beta/d}$. Then  the r.h.s. of \eqref{bound_int_a}  is bounded by
 $
  \|A^{-1}  \|  {R}/{(\de |t|)}.
 $
 So
 $$
 |\va|\gtrsim_A
  \frac{|\vu|- R\sqrt{|\vu|}}{|t|}\ge (1-|\vu|^{-1/6})\frac{|\vu|}{|t|} \ge
 \frac12\frac{|\vu|}{|t|}\ge \frac 12 |\vu|^{\beta/d}\ 
 $$
 since $|\vu|^{-1}\leq \alpha$, 
 if $\alpha$  is so small that $1-\alpha^{1/6}\geq 1/2$.
 Then ${\bf 1}_{{S_R}^c}(\va,t) \lappr_{A,d}   |\vu|^{-d/2+\beta/d}  |\va|^{d^2/2\beta-1}$, and  we deduce from
 \eqref{I(a,t)<} that for such values of $t$
 \begin{equation*}
 \begin{split}
 \left|\int_{\R^d} |I_{\va,t}| {\bf 1}_{{S_R}^c}(\va,t) d\va\right|&\lappr_{A,d} |\vu|^{-d/2+\beta/d}
 \int_{\R^{d}} |\va|^{d^2/2\beta-1} \sup_{|b_i|\le 1\,\forall i}|\tilde w(\de \vb +\va)|\, d\va
 \\
 &\lappr_{A,d{ ,m} } |\vu|^{-d/2+\beta/d} \|w\|_{0,K(d,\beta)}\,, 
 \end{split}
 \end{equation*}
 where $K(d,\beta)=d+\lceil d^2/2\beta\rceil+4$.
 On the other hand, by \eqref{p(t)<} with $M=0$, 
 $$
 \int_{|t|\le |\vu|^{1-\beta/d}}
  |p(t)|dt \lappr |\vu|^{1-\beta/d}\,,
 $$
 from which we obtain
 \begin{equation}\label{t<u}
 \int_{|t|\le |\vu|^{1-\beta/d}}dt\int_{\R^{d}}d\va \,|p(t)| |I_{\va, t}| {\bf 1}_{{S_R}^c}(\va,t) \lappr_{A,d{ ,m}}
 |\vu|^{-d/2+1} \| w\|_{0,K(d,\beta)}\,. 
 \end{equation}
 
 Putting together \eqref{t>u} and \eqref{t<u}  we get the assertion. 
 \qed

\subsubsection{End of the proof }  
 In order to complete the proof of Lemma~\ref{l:22} we combine 
 Lemmas~\ref{l:good_set} and \ref{l:bad_set} to get that 
 $$
 |\tilde I_q(\vc)|\lappr_{A,d,N{ ,m}} \left(\frac Lq R^{-N} +R^{d} |\vu|^{-d/2+1+\beta}\right)\left(
 \|w\|_{N,d+5}+\|w\|_{0,K(d,\beta)}\right)\,.
 $$
 We fix here $\beta = {\gamma_1}/2$, $R= |\vu|^{\frac{{\gamma_1}}{2d}}\le|\vu|^{\frac{1}{3}}$ 
 and pick $N= \lceil\tfrac{d^2}{{\gamma_1}}\rceil-2d>0$  (notice that $R>2\|A\|\sqrt{d}$ if
 $\alpha$ is small enough). 
 Then $K(d,\beta)=N+3d+4$, $R^{-N}\leq|\vu|^{-d/2+\gamma_1}\leq (L/q)^{-d/2+\gamma_1}$ since $-d/2+\gamma_1<0$ and $|\vu|\geq L/q$. Moreover, $R^d |\vu|^{-d/2+1+\beta}=|\vu|^{-d/2+1+\gamma_1}\leq (L/q)^{-d/2+1+\gamma_1}$. 
 This concludes the proof.
 \qed
\section*{Acknowledgements}
 The authors thank Professor Heath-Brown for his advice concerning his paper.

\appendix
\section{Function $\II(t)$.}\lbl{sec:I(t)}

In this appendix we assume that $d$ is an even number, $d=2d_1$, and the quadratic form $F$ has the form $F_0$. 

\subsection{The two measures on $\Sigma_t$.}

Recall that 
\be\non
 \Sigma_t=\{z=(x,y): x\cdot y=t \}\subset \R^{\dd}_x\times \R^{\dd}_y =\R^{2\dd}_z
\ee
 is a smooth manifold if $t\ne0$, while  $\Sigma_0$ has a locus at $(0,0)$. For $t\ne0$ we denote by 
 $dz\!\mid_{\Sigma_t}$ the volume element on $\Sigma_t$. If $t = 0$, then  we fist set $dz\!\mid_{\Sigma_0}$ to be 
  the volume element on $\Sigma_0 \setminus (0,0)$ and then extend it to a Borel measure on $\Sigma_0$ which assigns zero measure to the 
   locus  $(0,0)$.

We start with  a convenient disintegration of the measure $\dtz$. Let us set 
$
\Sigma_t^x = \Sigma_t\setminus\{(x,y): x=0\}.
$
Then for each  $t$ the mapping
\be\non
\pi_t^x: \Sigma_t^x \to \R^{\dd}\setminus\{0\}, \quad (x,y) \mapsto x,
\ee
is a  smooth affine euclidean  vector--bundle. For $t\ne0$ the set $ \Sigma_t^x $ equals  $\Sigma^x$, but not for $t=0$, and the quadric $\Sigma_0$ is
important for what follows.  Denote by $\sigma_t(x)$ the fibers of  $\pi_t^x$,
$
\sigma_t(x) = (\pi_t^x)^{-1}(x).
$
Then 
$$
\sigma_t(x) =(x,x^\perp + tx |x|^{-2}) \qquad \forall\, x\in \R^{\dd} \setminus \{0\},\;\; \forall\, t,
$$
where $x^\perp$ stands for the orthogonal complement to $x$ in $\R^{\dd}$.


\begin{proposition}\label{t_disintegr}  For any $t$ the 
 measure  $ dz\!\mid_{\Sigma_t} $, restricted to $\Sigma_t^x$, 
  disintegrates as follows:
\be\label{des1}
dz\!\mid_{\Sigma_t^x} = |x|^{-1} dx  |z| \,\dtx y,
\ee
where $\dtx y$ is the 
volume element on the  affine hyper space $\sigma_t(x)$.
\end{proposition}

We recall that equality \eqref{des1} means that for any bounded continuous function  $f$   on $\Sigma^x_t$  
\be\label{recall}
\int_{\Sigma^x_t} f(z)\dtz =\int_{\R^{\dd}\setminus \{0\}}  |x|^{-1}   dx\int_{\sigma_t(x)}  f(z) \, |z|\,\dtx y. 
\ee

\begin{proof}
The argument below follows the proof of Theorem 3.6 in \cite{DK1}. 
It suffices to verify \eqref{recall} for all continuous functions $f$, supported by a compact set $K$, for any
 $K\Subset  \Sigma_t \cap  \big
 (\R^{\dd}\setminus\{0\}) \times \R^{\dd} \big)$. 
For $x'\in \R^{\dd}\setminus\{0\}$  we denote $r'=|x'|>0$ and set
$
U_{x'} = \{x: |x-x' |<\tfrac12 r'\}$.
Since any $K$ as above can be covered by a finite system of domains $U_{x'} \times \R^{\dd}_y$, it suffices to prove \eqref{recall} for 
 any set 
  $U_{x'} \times \R^{\dd}_y =: U\subset \Sigma_t^x$ and any 
   $f\in C_0(U \cap\Sigma_t^x)$, where $C_0(\mathcal O)$ stands for  the space of continuous compactly supported functions on $\mathcal O$.  

Now we construct explicitly a trivialisation of the linear bundle $\pi_t^x$ over $U_{x'} $. To do this we fix in $\R^{\dd}$ a coordinate system,
corresponding to a frame $(e_1, \dots, e_d)$ such that the ray $\R_+ e_1$ intersects $U_{x'} $. Then 
\be\non
x_1\ge\kappa>0\quad\text{for any}\quad x=(x_1,x_2,\dots,x_d) =: (x_1, \bar x)\in  U_{x'} .
\ee
Next we construct 
   a linear in the  second argument $\bar\eta$
coordinate mapping 
$
\Phi_t: U_{x'}\times \R^{\dd-1} \to  (\pi_t^x)^{-1}U \cap  \Sigma^x_t = U \cap  \Sigma^x_t 
$
of the form
 $$
 \Phi_t(x, \bar\eta) = \big(x,  \Phi_t^x(\bar\eta)\big), 
 \qquad  \Phi_t^x(\bar\eta) =  ({\phi_t}(x, \bar\eta), \bar\eta).
 $$
 The function $\phi_t$  should be  such  that $\Phi_t^x(\bar\eta)\in \sigma_t(x)$. That is, it should satisfy
 $x\cdot \Phi_t^x(
 \bar \eta)= x_1 {\phi_t} + \bar x\cdot\bar\eta=t$.
 From here we find that 
 $\ 
 {\phi_t} = \frac{t-\bar x \cdot\bar\eta}{x_1} 
 $. 
 Thus obtained mapping $\Phi^x_t$ is affine in $\bar\eta$, and the image of $\Phi_t$  equals  $U \cap  \Sigma^x_t $.
 So $\Phi_t$ provides the required trivialisation of $\pi_t^x$ over $U \cap  \Sigma^x_t $. 
 
In the coordinates $(x, \bar\eta) \in U_{x'}\times\R^{\dd-1}$ the hypersurface $\Sigma^x_t$ is embedded in $\R^{2\dd}_z$ as a graph of the function $(x, \bar\eta) \mapsto \phi_t $. 
Accordingly  in these  coordinates 
  the volume element $\dtz$ on $\Sigma^x_t$ reeds $\dtz =\bar p_t(x,\bar\eta) dx\,d\bar\eta ,$ 
where
$$
 \bar p_t(x,\bar\eta)= \big(1+|\nabla {\phi_t} (x,\bar\eta) |^2\big)^{1/2}=
\Big( 1+ x_1^{-2}\big( x_1^{-2}(t-\bar x\cdot \bar\eta)^2 +|\bar\eta|^2 + | \bar x|^2 \big)
\Big)^{1/2} .
$$
So 
$$
\int_{U \cap  \Sigma^x_t } f(z) \dtz  = \int_{U_{x'} } \Big( \int_{\R^{\dd-1}} f(x, \Phi_t^x(\bar\eta)) \,\bar p_t(x, \bar\eta)\,d\bar\eta\Big)dx. 
$$
Passing from the variable $\bar\eta$ to $y= \Phi_t^x(\bar\eta)\in \sigma_t(x)$ we write the measure $\bar p_t(x, \bar\eta)d\bar\eta $ as
$p_t(z) d_{\sigma_t(x)} y$ with 
\be\non
 p_t(x,y)=\bar p_t(x, (\Phi_t^x)^{-1}y) |\det \Phi_t^x|^{-1}
\ee
(determinant of an affine map is understood as that of its linear part, mapping $\R_{\bar\eta}^{\dd-1}$ to $\sigma_t(x) = \R^{\dd-1}$).  Then 
\be\label{rela}
\int_{U \cap  \Sigma^x_t } f(z) \dtz  =  \int_{U_{x'}} \Big( \int_{ \sigma_t(x)} f(z) p_t(z) {\dtx} y \Big) dx.
\ee
The smooth function $p_t$ in the integral above is defined on $U \cap \Sigma^x_t$ 
in a unique way and does not depend on the trivialisation of $\Pi$ over 
${U_{x'}}$, used to obtain it. Indeed, if $p_1(z)$ is another continuous  function on $U \cap \Sigma^x_t$
 such that \eqref{rela} holds with $p_t:= p_1$, 
 then 
$\ 
\int dx  \int_{x^{\perp}} f(z) (p_t(z)-p_1(z))  {\dtx} y =0$ for any
$f\in C_0(U \cap  \Sigma^x_t ),$
which obviously implies that $p_1=p_t$.

 To establish \eqref{recall} it remains to verify that 
in \eqref{rela}
\be\label{remains}
p(x_*, y_*) = |x_*|^{-1} |z_*| \qquad \forall\, z_*= (x_*, y_*) \in U\cap\Sigma_t. 
\ee
To prove this equality 
let us choose in $\R^{\dd}$ the euclidean coordinates, corresponding to a frame  with the first basis vector 
 $e_1= x_*/|x_*|$. Then  $x_*=( |x_*|,0)$ and $y_*=(t/ |x_*|,\bar y_*)$
  for some  $\bar y_*\in\R^{\dd-1}$. Repeating the calculation above in this coordinate system we readily 
   see that 
   $
   \Phi_t^{x_*} (\bar\eta) = (t/|x_{*}| , \bar\eta).
   $
   So 
   $
   \phi_t(x_*, \bar\eta) = t/|x_*|, 
   $
   and
   $$
   \bar p_t(x_*, \bar\eta) = |x_*|^{-1} \big( |x_*|^2 + t^2 |x_*|^{-2} + |\bar\eta|^2\big)^{1/2}, \quad \bar\eta \in \R^{\dd-1}. 
   $$
   Since
   $
    \Phi_t^{x_*} (\bar\eta)  = (t/|x_*|, \bar\eta) \in \sigma_t(x_*),
   $
   then  $\det \Phi_t^{x_*}=1$. Thus  we have  that 
$$
\bar p_t(x_*, {\bar y_*}) =  p_t(x_*,   y_*) =
 | x_*|^{-1} ( | x_*|^2+ | y_*|^2)^{1/2} = |x_*|^{-1} |z_*|, 
$$
and \eqref{remains} is obtained.  This proves \eqref{des1}.
\end{proof}

Considering the projection $\pi^y_t: \Sigma_t\ni
(x,y)\mapsto y$ instead of $\pi_t^x$  we see that the volume element 
$\dtz$, restricted to  the domain $\Sigma^y_t=  \{(x,y)\in\Sigma_t: y\ne0\}$, disintegrates as
\be\label{Pii}
dz \mid_{\Sigma^y_t}= 
 |y|^{-1}dy\, |z| d_{(\pi^y_t)^{-1}y} x, \quad y\in \R^{\dd}\setminus\{0\}.
\ee
Since $\Sigma^y_t \cap \Sigma^w_t $ is the empty set if $t\ne0$ and is $(0,0)$ if $t=0$, then from here we obtain that 
\be\label{full_vol}
\text{
 $\Sigma_t^x$ has the full volume in $\Sigma_t$. }
 \ee
 As $d\ge2$, then  by \eqref{des1} the function $|z|^{-1}$ is locally integrable on $\Sigma_t^x$; so also on 
  $\Sigma_t$. Accordingly, the measure 
 $d\mu_t = |z|^{-1} dz\mid_{\Sigma_t}$ is well defined and  is  equivalent to $\dtz$. Using  \eqref{des1} and \eqref{Pii}
  we get:

\begin{corollary}\label{cor1}
For any $t$ the  measure   $\mu_t$    disintegrates as 
\be\label{desSigma}
(\mu_t\!\mid_{\Sigma_t^x})  (dz) = |x|^{-1}  dx\,\dtx y,
\ee
and as 
\be\label{desSigma1}
(\mu_t\!\mid_{\Sigma_t^y})  (dz) = |y|^{-1}  dy \, d_{(\pi^y_t)^{-1}y} x.
\ee
\end{corollary}

\subsection{Integrating over $\Sigma_t$.}

 	  Using the embedding $\Sigma_t \to \R^{2\dd}$  we regard $\mu_t$ both as an atomless  Borel measure on $\Sigma_t$ and on 
 $\R^{2\dd}$. Now our goal is to study integrals 
\be\label{I_int}
\II(t)=\II(t;f):=
\int_{\Sigma_t} f(z) \, \mu_t(dz), 
\ee
where $f$ is a ${k_f}$--smooth  function on $\R^{2\dd}$, $0\le k_f \le\infty$, decaying at infinity. 
  Due to \eqref{full_vol},  in \eqref{I_int} we may replace the integrating over 
 $\Sigma_t$ by that over $ \Sigma_t^x$ or over $ \Sigma_t^y$. So we may use the disintegrations \eqref{desSigma} and
 \eqref{desSigma1} to study $\II(t)$. To do this, note that for any $t$ the mapping 
 $$
 L_t= \Sigma_0^x \to \Sigma_t^x, \qquad (x,y) \mapsto (x, y+t|x|^{-2} x), 
 $$
 defines an affine isomorphism of the bundles $\pi_0$ and $\pi_t$, which we will also denote as $L_t$, and that $L_t$ 
  preserve the volume of the fibers. 
 So due to \eqref{desSigma} the mapping $L_t$ sends  the measure $\mu_0$ to $\mu_t$, and 
 \be\label{I_in}
 \II(t;f) = \int_{\Sigma_0} f(L_t(z)) \, \mu_0(dz) = \int_{\R^{\dd}} |x|^{-1} dx \int_{x^\perp} f(x, y+t|x|^{-2} x) d_{x^\perp} y,
 \ee
 since $\sigma_0(x)= x^\perp$. This equality suggests that the main difficulty to study $\II(t)$ comes from the integrating over small $x$'s. To separate it 
  from the effects, coming from the integrating over the vicinity of infinity we will split function $f$
 in a some of three functions. Firstly, taking any smooth function $\phi\ge0$ on $\R$ which vanishes for $|t|\ge2$ and equals one 
  for $|t|\le1$,  we write
 $$
 f = f_{00} + f_1, \qquad f_{00} = \phi(|z|^2) f. 
 $$
 Denoting $B_r(\R^n) =\{x\in \R^n: |x|\le r\}$ and  $B^r(\R^n) =\{x\in \R^n: |x|\ge r\}$ we see that 
 $\supp f_{00} \subset B_{\sqrt2}(\R^{2\dd})$ and $\supp f_1  \subset B^{1}(\R^{2\dd})$. So
 $$
 (x,y) \in \supp f_1 \ \Rightarrow\  |x| > 1/\sqrt2\;\  \text{or} \;\  |y| \ge 1/\sqrt2.
 $$
Then, setting 
$
f_{10} = f_1(z) \phi(4|x|^2)$ and $ f_{11}= f_1(z) ( 1-\phi(4|x|^2) ), 
$
we get that 
$\ 
f =   f_{00} + f_{10}+  f_{11},
$
where 
\be\label{supp}
\begin{split}
&\supp f_{11} \subset B^{1/2}( \R_x^{\dd})\times \R^{\dd}_y, \\
 &\supp f_{10} \subset  \Big(\big( B_{1/\sqrt2}(\R^{\dd}_x)\times \R^{\dd}_y \big) \cap \supp f_1\Big) \ \subset \R^{\dd}_x \times  B^{1/\sqrt2} (\R^{\dd}_y) . 
 \end{split}
\ee
We denote
$
\II_{ij}(t) = \II(t; f_{ij}).
$
Then $\II(t) = \II_{00}(t) + \II_{10}(t) + \II_{11}(t) $, and we will estimate the three integrals in the r.h.s.

For a function $F\in C^k(\R^N)$, $n_1 \in \NN, n_1\le k$  and $n_2\in \R$ 
 we set 
$
\| F\|_{n_1, n_2}  = \sup_{r} \max_{|\alpha|\le n_1} |\p_p^\alpha F(r)| \lan z\ran^{n_2}.
$
Assume that $ \| f\|_{k_f, n} <\infty$ for some $n\ge0$. Then obviously,
$$
\| f_{00} \|_{k_f, n}  \le C_{n} \| f\|_{k_f, 0}  \qquad 
\| f_{1j} \|_{k_f, n}  \le C_{k_f, n} \| f\|_{k_f, n}, \;\;  j=0,1.
$$
\smallskip

We start with the integral $ \II_{00}(t) $. It follows immediately from \eqref{desSigma} that this is a continuous function. 
Since for any $(x,y) \in \supp f_{00}$ we have $|(x,y)| \le \sqrt2$, then
$
| x\cdot y| \le |x| |y| \le\tfrac12 (|x|^2 +|y|^2) \le 1.
$
So  $ \II_{00}(t) $ vanishes for $|t|\ge1$. Now let $|t|<1$.  From \eqref{I_in} we get that 
\be\label{I1}
\begin{split}
\p^k \II_{00} (t) &= \int_{\R^{\dd}} |x|^{-1} dx \int_{x^\perp}(d^k/dt^k) f_{00}(x, y+t|x|^{-2} x)\, d_{x^\perp} y \\
&= \int_{\R^{\dd}} |x|^{-1} dx \int_{x^\perp} d_y^k f_{00}(x, y+t|x|^{-2} x) (|x|^{-2} x) \, d_{x^\perp} y .
\end{split}
\ee
Since $y\in x^\perp$, then 
$
|(x,  y+t|x|^{-2} x)|^2 = |x|^2 +|y|^2 +t^2 |x|^{-2}.
$
So the integrand of the internal integral is supported by the ball
$$
Q_t(|x|) =\{ y \in x^\perp : |y|^2 \le h_t(|x|^2), \qquad h_t(\rho) = 2-\rho- t^2\rho^{-1} \le2\;\; \text{for} \;\; \rho>0.
$$
The function $h_t$ of argument $\rho>0$  is  concave. For  $|t|<1$ it  has two positive zeroes 
$$
\rho_{1,2}(t) = 1 \pm \sqrt{1-t^2}, \qquad 0<\rho_1 = t^2/2 +O(t^4) < \rho_2 =2- t^2/2  +O(t^4) <2
$$
and is positive between them. Since $h_t\le2$, then 
$
\text{Vol} \, Q_t(|x|) \le C_{\dd-1} 2^{(\dd-1)/2}$ $= C'. 
$
The internal integral in \eqref{I1} is non-zero only when $h$ is positive, i.e. when
\be\label{hr}
\rho_1 < |x|^2 <\rho_2. 
\ee

Since $\rho_1>0$ for $t\ne0$, then we see from \eqref{I1} that  $\II_{00} $
is as smooth a function of $t\ne0$ as $f_{00}$ is; so  $\II_{00}\!\mid_{\R\setminus \{0\}}  \in C^{k_f}$. To study the behaviour of $\II_{00} $ at zero we note 
that since $ \|d_y^k f_{00} \| \le C_k \| f\|_{k,0}$ and $\text{Vol} \, Q_t(|x|) \le C'$, then by \eqref{I1}  and  \eqref{hr} 
$$
|\p^k \II_{00} (t) |\le C_k \| f\|_{k,0} \int_{\sqrt{\rho_1}}^{\sqrt{\rho_2}} r^{-1+\dd-1-k} \,dr.
$$ 
As ${\sqrt{\rho_1}} \ge c|t|$ and $\rho_2<2$, then 
\be\label{I3}
|\p^k \II_{00} (t)| \le C_k \| f\|_{k,0} \int_{c|t| }^{\sqrt{2}} r^{\dd-2-k} \le C_k \| f\|_{k,0} (1+ |t|^{\dd-k-1}), 
\ee
where we make an agreement that 
$
|t| ^0 := \max( \ln |t|^{-1}, 1). 
$

Naturally  estimate \eqref{I3}  remains true for $ \II (t; f)$ if $f$ is a $C^k$-function with a compact support. In general it cannot be improved for $C^k_0$-functions:


\begin{example}
Let $f$ be supported by the ball $B_1(\R^{2\dd})$, so that $f_{00} =f$. Further on, let $f= F(|x|^2) g(|y|^2)$, where $F$ and $g$ are non-negative 
$k_f$-smooth functions, supported by $ [-1/2, 1/2]$, and such that 
$0 \notin \supp g$.  Then
$$
f_t(x,y):= F(|x|^2) g(|y+t x|x|^{-2}|^2) = F(|x|^2) g(|y|^2 + t^2 |x|^{-2}). 
$$
So
$
\p_t f_t = F(|x|^2) g' (|y|^2 + t^2 |x|^{-2}) 2t |x|^{-2},
$
and 
$$
\p^k_t f_t = F(|x|^2)  \sum_{l=1} ^k C_l  g^{(l)}  (|y|^2 + t^2 |x|^{-2}) t^l |x|^{-l-k}, 
$$
where the coefficients $C_l$ are non-negative and some of them are zero.  Then the internal integral in \eqref{I1} equals 
$$
F(|x|^2)   \sum_{l=1} ^k C_l |t|^l |x|^{-l-k} J_l(|x|^2), \quad J_l (|x|^2) = \int  g^{(l)}  (|y|^2 + t^2 |x|^{-2}) d_{x^\perp} y.
$$
Now  we see from \eqref{I1} that 
$\p^k \II(t;f) $ equals to the sum in $l$ of the integrals 
$$
\Upsilon_l:=
C_l |t|^l \int_{c|t|} ^2 r^{-1+\dd-1-l-k} F(r^2) J_l(r^2)\, dr \sim \pm (1+ |t|^{\dd-k-1}).
$$
For generic $F$ and $g$ the numbers $\Upsilon_l$, corresponding to non-zero $C_l$, 
 do not vanish and do not cancel each other (see below), so in general case
  estimate \eqref{I3} cannot be improved.
\medskip

\noindent {\it Sub-example}. Let $d\ne3$ and $k=1$, so also $l=1$ (if $d=3$ we may consider $k=2$ and argue as below).
 Denoting $A= t^2 |x|^{-2}$ we have
\[
\begin{split}
J_1(|x|^2)& = C_{\dd-1} \int_0^\infty R^{\dd-2} g'( R^2 +A) \, dR\\
&= \frac1{2 C_{\dd-1}} \int_0^\infty R^{\dd-3} \frac{d}{dR}
 g( R^2 +A) \,dR
 \\&
  =- \frac{C_{\dd-1}(\dd-3)}{2}   \int_0^\infty R^{\dd-4}   g( R^2 +A) \,dR,
\end{split}
\]
which is $>0$ or $<0$, depending on $d$. Accordingly the integral in the expression for $\Upsilon_1$  does not vanish. 
\end{example} 
\bigskip

It remains to consider functions $\II_{11}$ and $\II_{10}$. Let us start with $\II_{11}$, assuming that
\be\label{cond}
n>2\dd-2. 
\ee
 In view of \eqref{supp},  $\p^k\II_{11}$ may be written in the 
form \eqref{I1} with $f_{00}$ replaced by $f_{11}$, where the integrating in $dx$ is taken over $B^{1/2} (\R^{\dd})$.  So $|x|\ge 1/2$ everywhere on the 
support of the integrand, and the function $\II_{11}$ is $k_f$-smooth. Since in the integral 
$$
\| d_y^k f_{11}\| \le C( 1+ |x|^2 +|y|^2 +t^2 |x|^{-2} )^{-n/2} \| f\|_{k,n} 
$$
and $|x| \ge 1/2$, then 
\be\label{I5}
\begin{split}
|\p^k \II_{11} (t&) | \le C  \| f\|_{k,n}   \int_{|x|\ge1/2} |x|^{-1-k} dx \int_{x^\perp}\! ( 1+ |x|^2 +|y|^2 +t^2 |x|^{-2} )^{-\frac{n}2}\, d_{x^\perp} y \\
& \le C'  \| f\|_{k,n} \int_{1/2}^\infty r^{-1-k +\dd-1}\,dr   \int_0^\infty\!( 1+ r^2  +t^2 r^{-2} +\rho^2 )^{-\frac{n}2} \rho^{\dd-2}\,  d\rho .
\end{split}
\ee
Denoting $1+r^2 +t^2 r^{-2} =A^2$ we write the internal integral as 
$$
\int_0^\infty( A^2+\rho^2 )^{-\frac{n}2} \rho^{\dd-2}\,  d\rho = A^{-n- 1 +d_1} \! \int_0^\infty\!( 1+|x|^2 )^{-\frac{n}2} x^{\dd-2}\,  dx= C_n  A^{-n-1 +d_1} 
$$
(we recall \eqref{cond}). So
$$
|\p^k \II_{11} (t)| \le   C  \| f\|_{k,n} \int_{1/2}^\infty r^{\dd-k-2 } (1+r^2 +t^2 r^{-2})^{-(n+1-d_1)/2}  \,dr   .
$$
For $|t|\le 1$ the integral in the r.h.s. obviously is bounded by a $t$-independent constant. Now let  $|t|\ge 1$.
Since by Young's inequality 
$$
(A+B)^{-1} \le C_a^{-1} A^{-a} B^{a-1}, \quad 0<a<1,
$$
for any $A,B>0$, and as  $r\ge1/2$, then 
$$
(1+r^2 + t^2 r^{-2} )^{-1} \le C'_a  r^{-2a} t^{2(a-1)} r^{-2(a-1)} = C'_a r^{-2(2a-1)} |t|^{2(a-1)}. 
$$
Denoting $2a-1 =b \in (-1, 1)$ we get that 
$$
|\p^k \II_{11} (t)| \le   C_{k,a}  \| f\|_{k,n} |t|^{-\frac{1-b}2 (n+1-d_1)} 
 \int_{1/2}^\infty r^{\dd-k-2 -b(n+1-d_1) }   \,dr   .
$$
Consider
$
b_* =\frac{\dd-k-1}{n+1-\dd} <1. 
$
Taking any $b\in (b_*,1)$ we achieve that the exponent for $r$ in the integral above is $<-1$, so the integral converges. Denote
$
\kappa(b) = \frac{1-b}2 (n+1-\dd). 
$
Then $\kappa(b_*) = \frac12 (n+k-2(\dd-1))>0$ and any $0< \kappa <\kappa(b_*)$ equals  $\kappa(b)$ for some $b\in (b_*,1)$.

We have seen that  for every $0<\kappa < \kappa(b_*)$,
\be\label{I6}
|\p^k \II_{11} (t)| \le   C(k,n, \kappa)   \| f\|_{k,n} \lan t \ran^{-\kappa} \quad \forall\, t.
\ee

For integral $ \II_{10}$ we use  disintegration \eqref{desSigma1} to get an analogy of estimate \eqref{I5} with $x$ and $y$ swapped. Since 
by \eqref{supp} 
$|y|\ge1/\sqrt2$ for $(x,y)$ in the support of $f_{10}$,  then repeating the argument above we get that 
 $ \p^k \II_{10} $ also satisfies estimate 
 \eqref{I6} with a modified constant $C_k$. 
We have proved 

\begin{proposition}\label{p_2}
Assume that $ \| f\|_{k,n}<\infty$ for some $k\in \NN$ and that  $n>2\dd-2$. Then for $0<|t| \le 1$ 
$$
| \p^k \II (t;f)| \le   C_k  \| f\|_{k,n} (1+ |t|^{\dd-k-1})
$$
(we recall the notation $|t|^0 := \max( \ln|t|^{-1},1)$), 
while for $|t| \ge1$ 
$$
| \p^k \II (t;f)| \le   C  \| f\|_{k,n}  | t |^{-\kappa}
$$
for every $\kappa< \tfrac12 (n+k-2(\dd-1))$ with a suitable $C$, depending on $k, n$, $\kappa$.
\end{proposition} 

We have seen that for a smooth $f$, decaying at infinity, the function $ \II (t;f)$ is $(\dd-2)$-smooth, is smooth outside zero 
 and decays at infinity. But for a generic $f$ its derivative of order $d_1-1$ has at zero a logarithmic singularity.

\end{document}